\documentclass[11pt,reqno]{article}
\usepackage{amsfonts, amsthm, amsmath, amssymb}
\usepackage{graphicx}
\usepackage[english]{babel}
\usepackage[cp1251]{inputenc}

\theoremstyle{plain}

\numberwithin{equation}{section}

\newtheorem {lemma} {Lemma}[section]
\newtheorem {theorem} {Theorem}
\newtheorem{proposition}{Proposition}[section]
\newtheorem*{thh}{Theorem}
\theoremstyle{remark}
\newtheorem{remark}{Remark}

\DeclareMathOperator{\tr}{tr}

\begin{document}

\title{Kerov's interlacing sequences and random matrices}

\author{Alexey Bufetov\thanks{Institute for Information Transmission Problems, Independent University of Moscow and Higher School of Economics, e-mail: \texttt{alexey.bufetov@gmail.com}} }

\date{}

\maketitle

\begin{abstract}
To a $N \times N$ real symmetric matrix Kerov assigns a piecewise linear function whose local minima are the eigenvalues of this matrix and whose local maxima are the eigenvalues of its $(N-1) \times (N-1)$ submatrix. We study the scaling limit of Kerov's piecewise linear functions for Wigner and Wishart matrices. For Wigner matrices the scaling limit is given by the Verhik-Kerov-Logan-Shepp curve which is known from asymptotic representation theory. For Wishart matrices the scaling limit is also explicitly found, and we explain its relation to the Marchenko-Pastur limit spectral law.
\end{abstract}

\section{Introduction}

Consider two sequences of real numbers $\{ x_i \}_{i=1}^n$, $\{ y_j \}_{j=1}^{n-1}$ such that
\begin{equation}
\label{interlace}
x_1 \ge y_1 \ge x_2 \ge \dots \ge x_{n-1} \ge y_{n-1} \ge x_{n}.
\end{equation}
We say that the sequences $\{ x_i \}_{i=1}^n$, $\{ y_j \}_{j=1}^{n-1}$ \emph{interlace}.
Define
\begin{equation*}
z_0 = \sum_{i=1}^n x_i - \sum_{j=1}^{n-1} y_j.
\end{equation*}

Following Kerov (see \cite{Ker}) we define a \emph{rectangular Young diagram} $w^{ \{x_i\}, \{y_j\} } (x)$ which is uniquely determined by the following conditions (see Figure \ref{separationYoungDiagr}):

\begin{figure}
\begin{center}
\includegraphics[height=6.5cm]{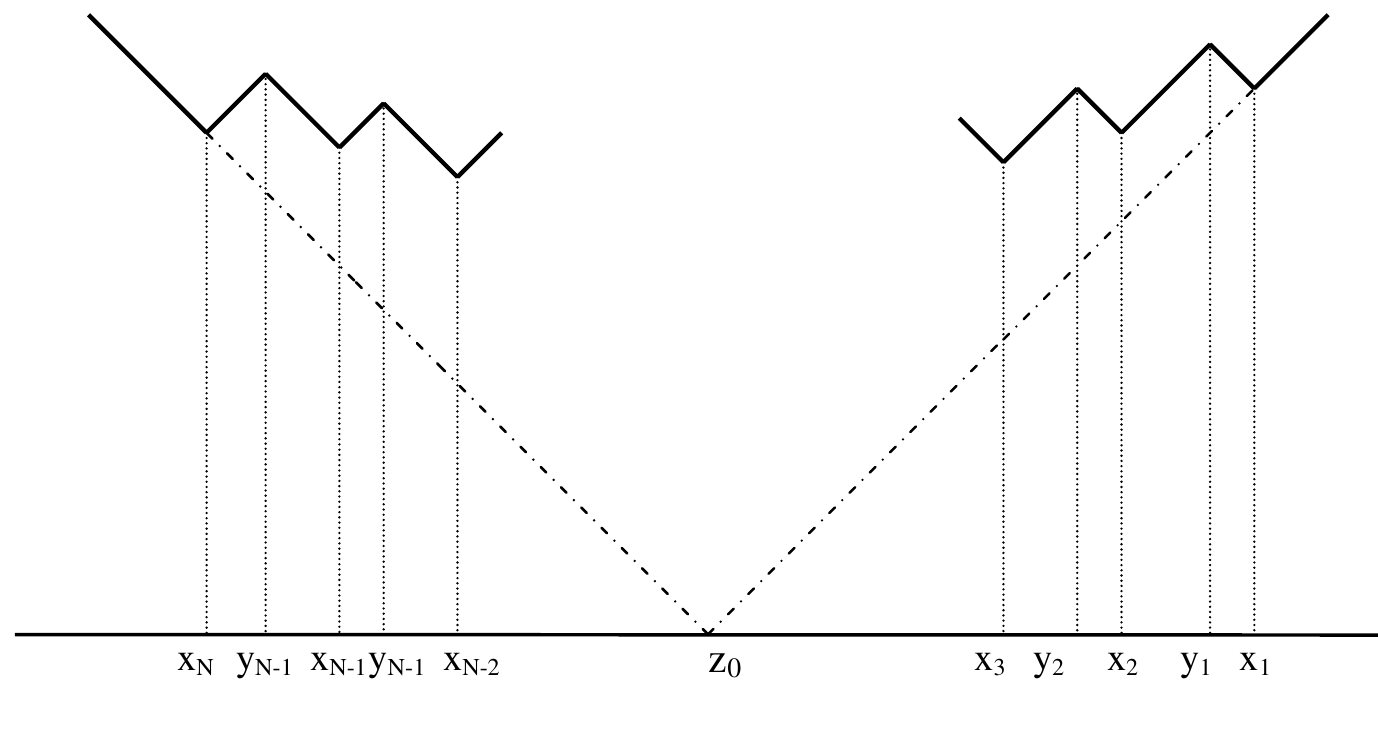}
\caption{A part of a rectangular Young diagram}
\label{separationYoungDiagr}
\end{center}
\end{figure}

1) $w^{\{x_i\},\{y_j\}} (x): \mathbb R \to \mathbb R$ is a continuous piecewise linear function, and $\dfrac{\partial}{\partial x} w^{\{x_i\},\{y_j\}} (x)= \pm 1$, except for finitely many points, which are exactly the local extrema of the function $w^{\{x_i\},\{y_j\}} (x)$.

2) $\{ x_i \}_{i=1}^n$ are local minima of $w^{\{x_i\},\{y_j\}} (x)$, $\{ y_j \}_{j=1}^{n-1}$ are local maxima of $w^{\{x_i\},\{y_j\}} (x)$, and there are no other local extrema.

3) $w^{\{x_i\},\{y_j\}} (x) = |x-z_0|$ when $|x|$ is large enough.

Let $S$ be a $N \times N$ real symmetric matrix. By $\hat S$ we denote its $(N-1) \times (N-1)$ submatrix;
it is obtained from $S$ by removing the $N$th row and column. It is well-known that the eigenvalues of $S$ and $\hat S$ interlace (see, e.g, \cite[p.185]{CauchyInter}).
Thus, to any symmetric matrix we can assign a rectangular Young diagram built from the eigenvalues of $S$ and $\hat S$.

Let $\{ Z_{ij} \}_{i,j=1}^{\infty}$ be a family of independent and identically distributed zero-mean, real-valued random variables such that $\mathbf E Z_{11}^2 =1$ and
\begin{equation*}
\mathbf E |Z_{11}|^k < \infty, \qquad \mbox{for all $k=1,2,3, \dots$.}
\end{equation*}
A symmetric $N \times N$ matrix $X_N$ with entries
\begin{equation*}
X_N(i,j) = X_N (j,i) = Z_{ij}, \qquad \mbox{for $i \le j$,}
\end{equation*}
is called a \emph{Wigner matrix}. Let $w_N^X (x)$ be the rectangular Young diagram constructed from the eigenvalues of $X_N$ and $\hat X_N$. Note that $w_N^X (x)$ is a random function. We are interested in the limit behaviour of $w_N^X (x)$ as $N \to \infty$.

Let
\begin{equation*}
\Omega(x) = \begin{cases}
\dfrac{2}{\pi} \left( x \arcsin (\frac{x}{2}) + \sqrt{4 - x^2} \right), \qquad & |x| \le 2, \\
|x|, \qquad & |x| \ge 2,
\end{cases}
\end{equation*}
be the Vershik-Kerov-Logan-Shepp curve (see \cite{VK1} and \cite{LS}).

\begin{theorem}
\label{th1}
As $N \to \infty$, we have
\begin{equation*}
\lim_{n \to \infty} \sup_{x \in \mathbb R} \left| \frac{1}{\sqrt{N}} w_N^X (x \sqrt{N}) - \Omega(x) \right| =0, \qquad \mbox{in probability.}
\end{equation*}
\end{theorem}

\begin{figure}
\begin{center}
\includegraphics[height=6.5cm]{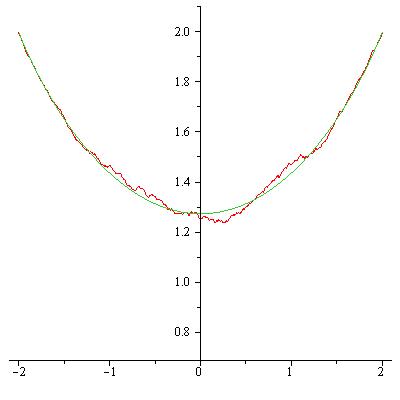}
\caption{The rectangular Young diagram $w_N(x)$ and $\Omega(x)$ for $N=1100$}
\label{separationWigner}
\end{center}
\end{figure}


Interlacing sequences arise naturally in several areas. They provide a useful parametrization of Young diagrams (see \cite{Ker4},\cite{IO}). They also appear as roots of two consecutive orthogonal polynomials (see \cite{Ker1}). A more general notion of \emph{interlacing measures} was studied in \cite{Ker3}.

For the first time the curve $\Omega(x)$ appeared from a representation theoretic problem. This curve is the limit shape of the random Young diagram distributed according to the Plancherel measure (see \cite{VK1}, \cite{LS}, \cite{VK2}, and \cite[Section 5]{IO} for more details).
Then it was found that $\Omega(x)$ is a scaling limit for separation of roots of orthogonal polynomials (see \cite{Ker1}). This curve also arises as the scaling limit in the evolution of continuous Young diagrams (see \cite{Ker2}) and in random matrix theory.
Let us formulate Kerov's result from \cite{Ker1} related to random matrix theory.


Let $h_N \subset \mathbb R^N$ be a random hyperplane such that $0 \in h_N$ and the normal vector to $h_N$ is uniformly distributed on the unit sphere. Let $p$ be the projection operator to $h_N$.
Regard $X_N$ as an operator in $\mathbb R^N$ and consider the operator $p X_N p$ in $h_N$.
The eigenvalues of $X_N$ and $p X_N p$ interlace. Let us construct the rectangular Young diagram $\tilde w_N$ from these eigenvalues as above.

\begin{thh}[\cite{Ker1}, Th. 3.6]
As $N \to \infty$, we have
\begin{equation}
\label{KerovTh}
\lim_{n \to \infty} \frac{1}{\sqrt{N}} \, \mathbf E \tilde w_N (x \sqrt{N}) = \Omega(x),
\end{equation}
and the limit is uniform on $x \in \mathbb R$.
\end{thh}

\begin{remark}
In the context of Theorem \ref{th1} we consider the restriction to a fixed hyperplane
while here the hyperplane is random. Another difference is that Theorem \ref{th1} establishes the convergence in probability while \eqref{KerovTh} gives only the convergence of mean.
\end{remark}

Let us proceed to Wishart matrices. Let $M=M(N)$ be a sequence of positive integers such that
\begin{equation*}
\lim_{N \to \infty} \frac{M(N)}{N} = \alpha \ge 1.
\end{equation*}
Let $\mathcal W_N$ be a $N \times M(N)$ matrix with i.i.d. entries of mean zero and variance 1, and such that
\begin{equation*}
\mathbf E \mathcal |W_N(1,1)|^k < \infty, \qquad \mbox{for all $k=1,2,3, \dots$}.
\end{equation*}
By $Y_N$ we denote the $N \times N$ Wishart matrix $Y_N = \mathcal W_N \mathcal W_N^t$. Let $w^Y_N (x)$ be the rectangular Young diagram which is defined by the interlacing eigenvalues of $Y_N$ and $\hat Y_N$.

Let us define a continuous function $\Omega_{\alpha} : \mathbb R \to \mathbb R$. Set
\begin{equation*}
\Omega''_{\alpha} (x) = \frac{x+(\alpha - 1)}{\pi x \sqrt{ 4 \alpha - (x-(\alpha+1))^2}}, \qquad x \in [(\alpha+1) - 2 \sqrt{\alpha}; (\alpha+1) + 2 \sqrt{\alpha}],
\end{equation*}
and
\begin{equation*}
\Omega_{\alpha} (x) = |x-\alpha|, \ \ \ \ x \in (-\infty;(\alpha+1) - 2 \sqrt{\alpha}] \cup [(\alpha+1) + 2 \sqrt{\alpha}; \infty).
\end{equation*}
These formulas determine the function $\Omega_{\alpha} (x)$ uniquely.

\begin{theorem}
\label{th2}
As $N \to \infty$, we have
\begin{equation*}
\lim_{n \to \infty} \sup_{x \in \mathbb R} \left| \frac{1}{N} w^Y_N(x N) - \Omega_{\alpha} (x) \right| =0, \qquad \mbox{in probability.}
\end{equation*}
\end{theorem}

\begin{figure}
\begin{center}
\includegraphics[height=6.5cm]{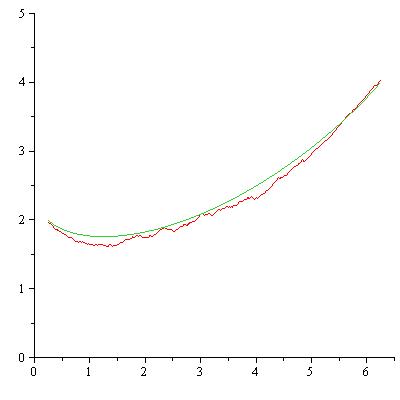}
\caption{The rectangular Young diagram $w_N^Y (x)$ and $\Omega_{\alpha} (x)$ for $\alpha=2.25$ and $N=400$}
\label{separationWishart}
\end{center}
\end{figure}

\begin{remark}
The limit shapes $\Omega_{\alpha} (x)$ are closely connected with Biane's limit shapes (see \cite{Biane}). This connection is described in Section 5.
\end{remark}

This paper is organized as follows. In Section 2 we give some preliminaries. In Section 3 we prove Theorem \ref{th1}. In Section 4 we prove Theorem \ref{th2}. In Section 5 we describe a link between these limit shapes and semicircle and Marchenko-Pastur limit laws.

\textbf{Acknowledgments.}

The author is grateful to G.~Olshanski and A.~Borodin for valuable discussions. The author is grateful to V.~Gorin and L.~Petrov for valuable comments. The author was partially supported by Simons Foundation-IUM scholarship, by Moebius Foundation scholarship, and by RFBR-CNRS grants 10-01-93114 and 11-01-93105.

\section{Continual Young diagrams}
A \emph{continual Young diagram} (see \cite{Ker}) is a function $w(x)$ on $\mathbb R$ such that

1) $|w(x_1) - w(x_2)| \le |x_1-x_2|$ for any $x_1, x_2 \in \mathbb R$.

2) There exists a point $x_0 \in \mathbb R$, called the \emph{center} of $w$, such that $w(x) = |x-x_0|$ when $|x|$ is large enough.

The set of all continual Young diagrams is denoted by $\mathcal D$. For any $w \in \mathcal D$ we define a function $\sigma(x)$
\begin{equation*}
\sigma(x) = \frac{1}{2} (w(x)-|x|).
\end{equation*}
Since $\sigma(x)$ satisfies the Lipschitz condition 1), its derivative $\sigma'(x)$ exists almost everywhere and satisfies $|\sigma'(x)| \le 1$. Note that $\sigma'(x)$ is compactly supported. The function $w(x)$ is uniquely determined by $\sigma'(x)$. Moreover, $w(x)$ is uniquely determined by the second derivative $\sigma''(x)$, which is understood in the sense of distribution theory.

Let us define the function $\tilde p_k: \mathcal D \to \mathbb R$, $k \in \mathbb N$, by setting
\begin{equation*}
\tilde p_k (w) = -k \int_{-\infty}^{\infty} x^{k-1} \sigma'(x) dx =
\int_{-\infty}^{\infty} x^k \sigma''(x) dx.
\end{equation*}

It is easy to see that for the rectangular Young diagram $w^{\{ x_i \},\{ y_j \}} (x)$ (see Introduction) we have
\begin{equation*}
\sigma''(x) = \sum_{i=1}^{n} \delta(x- x_i) - \sum_{j=1}^{n-1} \delta(x - y_j) - \delta(x).
\end{equation*}

We will need the following fact.

\begin{lemma}[\cite{IO}, Lemma 5.7]
\label{topology}
Let $\mathcal F([a;b])$ be the set of all real-valued functions $f(x)$ supported on the interval $[a,b] \in \mathbb R$ and satisfying the Lipschitz condition $|f(x_1)-f(x_2)| \le |x_1-x_2|$.

On the set $\mathcal F([a;b])$, the weak topology defined by the functionals
\begin{equation*}
f(x) \to \int_{x \in [a;b]} f(x) x^k dx, \ \ \ k=0,1,2, \dots,
\end{equation*}
coincides with the uniform topology.
\end{lemma}

\section{Proof of Theorem \ref{th1}}

\begin{lemma}
\label{lemma31}
Let $w_N^X$ be the (random) rectangular Young diagram defined by the eigenvalues of Wigner matrices $X_N$ and $\hat X_N$ (see Introduction). Then
\begin{equation}
\label{lem3-1}
\frac{\tilde p_k (w_N^X)}{N^{k/2}} \xrightarrow[N \to \infty]{} \begin{cases}
0, \ \ \ \mbox{$k$ is odd,}  \\
\frac{k!}{ (k/2)! (k/2)!}, \ \ \ \mbox{$k$ is even,}
\end{cases}
\end{equation}
where the convergence is in probability.
\end{lemma}

\begin{proof}
The proof of this Lemma closely follows the proof of Wigner's theorem (see, e.g., \cite[Section 2.1]{AGZ}) and is based on the well-known moment method.

Let $\{ \lambda_i^N \}_{i=1}^N$ be the eigenvalues of $X_N$ and let $\{ \lambda_i^{N-1} \}_{i=1}^{N-1}$ be the eigenvalues of $\hat X_N$. We have
\begin{equation*}
\tilde p_k (w_N^X) = \int_{x \in \mathbb R} x^k \left( \sum_{i=1}^N \delta \left( \lambda_i^N \right) - \sum_{j=1}^{N-1} \delta \left( \lambda_j^{N-1} \right) \right) dx = \tr \left(X_N^k \right) - \tr \left( X_{N-1}^k \right).
\end{equation*}
Let $\mathbf i_N = (i_1,i_2, \dots, i_k)$ range over the set of multi-indices such that $1 \le i_1, i_2, \dots, i_k \le N$. Likewise, let $\mathbf i_{N-1} = (i_1,i_2, \dots, i_k)$ range over the set of multi-indices such that $1 \le i_1, i_2, \dots, i_k \le N-1$. Then
\begin{multline*}
\tr \left(X_N^k \right) - \tr \left( X_{N-1}^k \right) = \sum_{\mathbf i_N} X_N(i_1,i_2) X_N(i_2, i_3) \dots X(i_k, i_1) \\ - \sum_{\mathbf i_{N-1}} X_N(i_1,i_2) X_N(i_2, i_3) \dots X(i_k, i_1) = \sum_{\mathbf i_N: N \in \mathbf i_N} X_N(i_1,i_2) X_N(i_2, i_3) \dots X(i_k, i_1),
\end{multline*}
where the last sum is taken over indices $\mathbf i_N = (i_1, i_2, \dots, i_k)$ such that there exists $r$, $1 \le r \le k$, satisfying $i_r=N$.

First, we compute
\begin{equation*}
\mathbf E \left( \sum_{\mathbf i_N: N \in \mathbf i_N} X_N(i_1,i_2) X_N(i_2, i_3) \dots X(i_k, i_1) \right).
\end{equation*}
This sum can be written as a sum of terms corresponding to suitably defined graphs that are in their turn associated to words. Suppose $k$ is odd; then the same estimates as in \cite[Lemma 2.1.6]{AGZ} show that the contribution to the degree $n^k$ is equal to 0.

Suppose $k$ is even; then the main contribution is given by the so called \emph{Wigner words} (see \cite[Def. 2.1.10]{AGZ}). The number of Wigner words is equal to $\dfrac{k!}{(k/2+1)! (k/2)!}$. The only difference of our case with the case of Wigner's theorem is that one vertex of a graph should be labeled by the special symbol $N$. This condition gives an extra factor $(k/2+1)$. Thus, we obtain
\begin{multline}
\label{3-1}
\lim_{N \to \infty} N^{-k/2} \mathbf E \left( \sum_{\mathbf i_N: N \in \mathbf i_N} X_N(i_1,i_2) X_N(i_2, i_3) \dots X(i_k, i_1) \right) \\ = \begin{cases}
0, \ \ \ \mbox{$k$ is odd,}  \\
\dfrac{k!}{ (k/2)! (k/2)!}, \qquad \mbox{$k$ is even.}
\end{cases}
\end{multline}
Secondly, we have
\begin{equation}
\label{3-2}
\lim_{N \to \infty} \left( N^{-k/2} \sum_{\mathbf i_N: N \in \mathbf i_N} X_N(i_1,i_2) X_N(i_2, i_3) \dots X(i_k, i_1) \right)^2 = 0.
\end{equation}
Indeed, this equality can be proved in the same way as in \cite[Lemma 2.1.7]{AGZ}.

From \eqref{3-1} and \eqref{3-2} it follows that $\tilde p_k$ converges to the right-hand side of \eqref{lem3-1} in $L^2$ and, consequently, in probability.

\end{proof}

\begin{lemma}[\cite{IO} Prop. 5.3]
\label{lemma32}
We have
\begin{equation*}
\tilde p_k (\Omega) = \begin{cases} 0, \qquad \mbox{$k$ is odd,}  \\
\dfrac{k!}{ (k/2)! (k/2)!}, \qquad \mbox{$k$ is even.}
\end{cases}.
\end{equation*}
\end{lemma}

\begin{lemma}
\label{lemOutZero}
There exists an interval $[-B;B]$ such that the probability that the inequalities
\begin{equation*}
-B < \lambda^N_N \le \dots \le \lambda_1^N < B
\end{equation*}
hold tends to 1 as $N \to \infty$.
\end{lemma}
\begin{proof}
This is a well-known fact from random matrix theory, see, e.g. \cite{AGZ}.
\end{proof}

Let $z_N$ be the center of the rectangular Young diagram $w_N^X$.
It is clear that
\begin{equation*}
\lim_{N \to \infty} \frac{z_N}{\sqrt{N}} = \lim_{N \to \infty} \frac{X(N,N)}{\sqrt{N}} = 0, \qquad \mbox{in probability.}
\end{equation*}
This equality and Lemma \ref{lemOutZero} imply the convergence of $w_N^X$ to $\Omega(x)$ outside the interval $[-B;B]$.

From Lemmas \ref{lemma31} and \ref{lemma32} for any $k \in \mathbb N$ we have
\begin{equation}
\label{eqPk}
\lim_{N \to \infty} \frac{\tilde p_k (w_N^X)}{N^{k/2}} = \tilde p_k (\Omega), \qquad \mbox{in probability.}
\end{equation}
The convergence of $w_N^X$ to $\Omega(x)$ inside the interval $[-B;B]$ follows from \eqref{eqPk} and Lemma \ref{topology}.

\section{Proof of Theorem \ref{th2}}

\begin{lemma}
\label{lemma41}
Let $w_N^Y$ be the (random) rectangular Young diagram defined by the eigenvalues of Wishart matrices $Y_N$ and $\hat Y_N$ (see Introduction). Then the following limit exists
\begin{equation}
\label{lem4-1}
\frac{\tilde p_k (w_N^Y)}{N^k} \xrightarrow[N \to \infty]{} m_k, \qquad \mbox{in probability.}
\end{equation}
The generating function of $m_k$ is given by the equation
\begin{equation*}
G_{\alpha} (z) := 1+\sum_{k=1}^{\infty} m_k z^k = \frac{1}{2} \left( 1+ \frac{(\alpha-1)z +1}{\sqrt{(\alpha-1)^2 z^2 - 2 (\alpha+1) z +1}} \right).
\end{equation*}
\end{lemma}

\begin{proof}
The proof of this Lemma is based on moment method and follows that given in \cite[Exercise 2.1.18]{AGZ}.

Let $\mathbf i_N = (i_1, i_2,\dots, i_N)$, $\mathbf j_N = (j_1, j_2, \dots, j_N)$ range over the set of multi-indices such that $1 \le i_1, i_2, \dots, i_N \le N$, $1 \le j_1, j_2, \dots, j_N \le M(N)$.
We have
\begin{multline*}
\tilde p_k (w_N^Y) = \tr(Y_N^k) - \tr(Y_{N-1}^k) = \sum_{\mathbf i_N: N \in \mathbf i_N} Y_N(i_1,i_2) Y_N(i_2,i_3) \dots Y_N(i_N, i_1) \\ = \sum_{\mathbf i_N, \mathbf j_N : N \in \mathbf i_N} \mathcal W_N (i_1, j_1) \mathcal W_N (i_2, j_1) \mathcal W_N (i_2, j_1) \mathcal W_N (i_2, j_2) \dots \mathcal W_N (i_k, j_k) \mathcal W_N (i_1, j_k),
\end{multline*}
where the condition $N \in \mathbf i_N$ means, as before, that there exists $r$ such that $i_r =N$.

The limit relation
\begin{equation*}
\lim_{N \to \infty} \frac{\mathbf E (\tr(Y_N^k) - \tr(Y_{N-1}^k))^2}{N^{2k}} = 0
\end{equation*}
can be proved in the same way as in \cite[Section 2.1]{AGZ}.
Thus, the only thing we need is to compute the leading term in
\begin{equation}
\label{WHmean}
\mathbf E \left( \sum_{\mathbf i_N, \mathbf j_N : N \in \mathbf i_N} \mathcal W_N (i_1, j_1) \mathcal W_N (i_2, j_1) \mathcal W_N (i_2, j_1) \mathcal W_N (i_2, j_2) \dots \mathcal W_N (i_k, j_k) \mathcal W_N (i_1, j_k) \right).
\end{equation}

Let us recall that a \emph{Dyck path} $\mathcal D_l$ of length $2 l$ is an integer-valued sequence $\{ S_n \}_{0 \le n \le 2l}$ such that $S_0=0$, $S_{2l}=0$, $|S(i)-S(i-1)|=1$ for $1 \le i \le 2l$, and $S(i) \ge 0$ for $0 \le i \le 2l$. Let $a(\mathcal D_l)$ be the number of indices $i$ such that $S(i)-S(i-1)=-1$ and $S(i-1)$ is odd and let $b( \mathcal D_l)$ be the number of indices $i$ such that $S(i)-S(i-1)=-1$ and $S(i-1)$ is even. We have $a(\mathcal D_l)+ b(\mathcal D_l)=l$.

There is a bijective correspondence between Wigner words and Dyck paths (see \cite[Section 2.1]{AGZ}).
It is easy to see that the main contribution to \eqref{WHmean} has order $N^{k}$ and is given by Wigner words of length $2k$ or, equivalently, by Dyck paths of length $2k$. If the sequence $(i_1, j_1, i_2, j_2, \dots, i_k, j_k)$ is a Wigner word then $a(\mathcal D_{k})$ is equal to the number of distinct $j$-indices in the tuple and $b(\mathcal D_{k})+1$ is equal to the number of distinct $i$-indices in the tuple, where $\mathcal D_{k}$ is the corresponding Dyck path. Therefore, each word gives a contribution $(b(\mathcal D_{k})+1) \alpha^{a(\mathcal D_{k})}$ and we obtain
\begin{equation}
\label{eqMk}
m_k = \sum_{\mbox{all } \mathcal D_{k}} (b(\mathcal D_{k}) +1) \alpha^{a(\mathcal D_{k})}.
\end{equation}

Let $\beta$ be a formal variable, and let
\begin{gather*}
d_{r} := \sum_{\mbox{all } \mathcal D_{r}} \beta^{r-a(\mathcal D_{r})} \alpha^{a(\mathcal D_{r})}, \qquad r \ge 1, \qquad d_0=1; \\
e_r := \sum_{\mbox{all } \mathcal D_{r}} \beta^{a(\mathcal D_{r})} \alpha^{r-a(\mathcal D_{r})},
\qquad r \ge 1, \qquad e_0=1.
\end{gather*}
Considering the moment of the first return of Dyck path $D_r$ to zero we obtain
\begin{equation*}
d_r = \alpha \sum_{j=1}^r d_{r-j} e_{j-1}, \qquad e_r = \beta \sum_{j=1}^r e_{r-j} d_{j-1}.
\end{equation*}
Hence,
\begin{equation*}
d_r = \frac{\alpha}{\beta} e_r, \qquad r \ge 1.
\end{equation*}
We obtain
\begin{equation}
\label{recEq}
d_r = \beta \sum_{j=2}^r d_{r-j} d_{j-1} + \alpha d_{r-1}.
\end{equation}
Let $d(z)$ be the generating function of $\{ d_r \}$:
\begin{equation*}
d(z):= 1+ \sum_{r=1}^{\infty} d_r z^r.
\end{equation*}
Using \eqref{recEq}, we get
\begin{equation*}
d(z) = 1 + \beta z d(z)^2 + (\alpha - \beta) z d(z).
\end{equation*}
Solving this equation and choosing the sign from the condition $d(0)=1$, we obtain
\begin{equation*}
d(z) = \frac{1 - (\alpha-\beta)z - \sqrt{((\alpha-\beta)z - 1)^2 - 4 \beta z}}{2 \beta z}.
\end{equation*}
From \eqref{eqMk}, we get
\begin{equation*}
1+ \sum_{k=1}^{\infty} m_k z^k = \frac{\partial}{\partial \beta} \left. \left( \beta d(z) \right) \right|_{\beta=1}.
\end{equation*}
This completes the proof of Lemma \ref{lemma41}.
\end{proof}

For $\alpha >1$ let $\Omega_{\alpha} (x)$ be defined by the formula
\begin{equation*}
\Omega_{\alpha} (x) := \begin{cases}
\dfrac{1}{\pi} \left( (2 \alpha - x) \arcsin \left( \frac{\alpha+1-x}{2 \sqrt{\alpha}} \right) - \arctan \left( \frac{(\alpha-1)^2 - x(\alpha+1)}{(\alpha-1) \sqrt{(x-\alpha)^2 + 2 \alpha +2 x -1}} \right) \right. \\ \left. + \sqrt{2 \alpha + 2 x -1 -(x-\alpha)^2} \right), \qquad \mbox{ for $x \in [\alpha +1 - 2 \sqrt{\alpha}; \alpha+1 + 2 \sqrt{\alpha}]$}; \\
|x-\alpha|, \qquad \mbox{ for $x \in (-\infty;(\alpha+1) - 2 \sqrt{\alpha}] \cup [(\alpha+1) + 2 \sqrt{\alpha}; \infty)$},
\end{cases}
\end{equation*}
and for $\alpha=1$ let
\begin{equation*}
\Omega_1 (x) := \begin{cases}
\dfrac{1}{\pi} \left( (x-2) \arcsin (\dfrac{1}{2} x-1) - \sqrt{4 - (x-2)^2} \right) + \dfrac{x}{2},
\qquad \mbox{for } x \in [0;2]; \\
|x-1|, \qquad \mbox{for } x \in (-\infty; 0] \cup [2; \infty).
\end{cases}
\end{equation*}

It is readily seen that for $\alpha \ge 1$, $\Omega_{\alpha} (x)$ is a continual Young diagram with the center $\alpha$.

\begin{lemma}
\label{omegaAlpha}
The generating function of $\{ \tilde p_k( \Omega_{\alpha}) \}$ is given by the formula
\begin{equation*}
1+ \sum_{k=1}^{\infty} \tilde p_k (\Omega_{\alpha}) z^k = \frac{1}{2} \left( 1+ \frac{(\alpha-1)z +1}{\sqrt{(\alpha-1)^2 z^2 - 2 (\alpha+1) z +1}} \right).
\end{equation*}
\end{lemma}
\begin{proof}
We recall that
\begin{equation*}
\tilde p_k (\Omega_{\alpha}) = \int_{\mathbb R} x^k \frac{\Omega''_{\alpha} (x)}{2} d x, \qquad k \in \mathbb N.
\end{equation*}
It can be verified by a direct computation that for $\alpha \ge 1$
\begin{equation*}
\Omega''_{\alpha} (x) = \begin{cases}
\dfrac{x+(\alpha - 1)}{\pi x \sqrt{ 4 \alpha - (x-(\alpha+1))^2}}, \qquad x \in [\alpha +1 - 2 \sqrt{\alpha}; \alpha+1 + 2 \sqrt{\alpha}], \\
0, \qquad x \ne [\alpha +1 - 2 \sqrt{\alpha}; \alpha+1 + 2 \sqrt{\alpha}].
\end{cases}
\end{equation*}
Using the technique of Stieltjes transform (see, e.g., \cite[Section 2.4]{AGZ}) we obtain the statement of the Lemma.
\end{proof}

Let $z_N^Y$ be the center of the diagram $w_N^Y$. Note that
\begin{equation*}
\frac{z_N^Y}{N} = \frac{Y_N (N,N)}{N} = \frac{\sum_{j=1}^N \mathcal W(N, j) \mathcal W(j, N)}{N} \xrightarrow[N \to \infty]{} \alpha, \qquad \mbox{in probability.}
\end{equation*}
This fact implies that the center of $\frac{1}{N} w_N^Y(N x)$ converges to the center of $\Omega(\alpha)$. From Lemmas \ref{lemma41} and \ref{omegaAlpha} we have
\begin{equation*}
\frac{\tilde p_k (w_N^Y)}{N^k} \xrightarrow[N \to \infty]{} \tilde p_k (\Omega_{\alpha}), \qquad \mbox{in probability.}
\end{equation*}
Also it is known that the spectrum of Wishart matrix is supported by a fixed compact interval with probability close to 1.
Theorem \ref{th2} follows from these facts in the same way as in the proof of Theorem \ref{th1}.

\section{Connection with semicircle and Marchenko-Pastur laws}
In this Section we briefly describe a connection between limiting shapes $\Omega$, $\Omega_{\alpha}$ and well-known semicircle and Marchenko-Pastur distributions, respectively.

For an interval $I$ by $\mathcal M(I)$ we denote the set of probability measures which are supported by $I$. For a measure $\mu \in \mathcal M(I)$ let
\begin{equation*}
\mu_k := \int_I x^k d \mu, \qquad k \in \mathbb N.
\end{equation*}
By $\mathcal D(I)$ we denote the set of continual Young diagrams such that $\sigma'' (x)$ is supported by $I$.

\begin{lemma}
There is a bijective correspondence $\mu \to w$ between $\mathcal M(I)$ and $\mathcal D(I)$.
It is characterized by the relation
\begin{equation}
\label{KerovTransform}
1+ \sum_{k=1}^{\infty} \mu_k z^k = \exp \left( \sum_{k=1}^{\infty} \frac{\tilde p_k (w)}{k} z^k \right).
\end{equation}
\end{lemma}
\begin{proof} See \cite{Ker}, \cite{Ker3}. \end{proof}
The measure $\mu$ is called the \emph{transition measure} of $w$. For more details about this correspondence, see \cite{Ker}, \cite[Section 8]{IO}.

A direct computation of the left- and right-hand sides of \eqref{KerovTransform} leads to the following Proposition.

\begin{proposition}

a) The semicircle law is the transition measure of $\Omega$.

b) The Marchenko-Pastur distribution with parameter $\alpha$ is the transition measure of $\Omega_{\alpha}$.
\end{proposition}

The item a) was first noted in \cite{Ker}.

In \cite{Biane} Biane introduced a family of curves which appeared as scaling limits in some problem of asymptotic representation theory. As shown in \cite{Mel}, the transition measure of every Biane's curve coincide with a Marchenko-Pastur distribution, within a homothetic transform. Thus the Biane's curves and the curves $\Omega_{\alpha}$ are closely related.

\begin{remark}
Let $\mathbf x_n = ( x_1, x_2, \dots, x_n)$ be a sequence of real numbers and suppose that for every $n$ the sequences $\mathbf x_n$ and $\mathbf x_{n-1}$ interlace. Let $\mu$ be a probability measure on $\mathbb R$ and let $\delta(x)$ be the Dirac measure at $x \in \mathbb R$. Assume that
\begin{equation*}
\frac{1}{n} \sum_{i=1}^n \delta(x_i) \xrightarrow[n \to \infty]{} \mu, \qquad \mbox{in the weak topology.}
\end{equation*}
It was shown in \cite{Ker1} that there are sequences with different limiting measures $\mu$ but with
the same scaling limit of $w^{\mathbf x_n, \mathbf x_{n-1}}$.
Therefore, the convergence of $\frac{1}{n} \sum_{i=1}^n \delta(x_i)$ to the measure $\mu$ does not imply the convergence of $w^{\mathbf x_n, \mathbf x_{n-1}}$
to the continual Young diagram with the transition measure $\mu$.
\end{remark}

\end{document}